\documentclass[oneside,italian,english,british]{article}

\usepackage[T1]{fontenc}
\usepackage[latin9]{inputenc}
\usepackage[a4paper]{geometry}
\usepackage{lmodern}
\usepackage{amsmath}
\usepackage{amssymb}
\usepackage{setspace}
\usepackage{esint}
\usepackage{babel}
\usepackage{tikz,fp,ifthen,fullpage}
\usetikzlibrary{backgrounds}
\usetikzlibrary{decorations.pathmorphing,backgrounds,fit,calc,through}
\usetikzlibrary{arrows}
\usetikzlibrary{shapes,decorations,shadows}
\usetikzlibrary{fadings}
\usetikzlibrary{patterns}
\usetikzlibrary{mindmap}
\usetikzlibrary{decorations.text}
\usetikzlibrary{decorations.shapes}

\makeatletter
\usepackage{amsthm}
\allowdisplaybreaks[1]
\makeatother

\newtheorem{thm}{Theorem}[section]
\newtheorem{lem}[thm]{Lemma}
\newtheorem{prop}[thm]{Proposition}
\newtheorem{defn}[thm]{Definition}
\newtheorem{cor}[thm]{Corollary}
\newtheorem{rem}[thm]{Remark}
\newtheorem{fig}[thm]{Fig.}

\newtheorem{ackn}{Acknowledgments}

\title{Evolution of spoon-shaped networks}
\author{Alessandra Pluda \thanks {Dipartimento di Matematica, Universit\`{a} di Pisa, 
Largo Bruno Pontecorvo 5, 56127, Pisa, Italy } }
\date{ }

\begin{document}

\maketitle
\begin{abstract} 
\noindent We consider a regular embedded network  composed by two curves, 
one of them closed, in a convex domain $\Omega$. 
The two curves meet only in one point, forming angle of $120$ degrees.
The non-closed curve has a fixed end point on $\partial\Omega$.
We study the evolution by curvature of this network.
We show that the maximal existence time depends only on the area enclosed 
in the initial loop, if the length of the non-closed curve stays bounded from below during the evolution.
Moreover, the closed curve 
shrinks to a point and 
the network is asymptotically approaching,
after dilations and extraction of a subsequence,
a Brakke spoon.
\end{abstract}


\section{Introduction}
In this paper we study the evolution by curvature of two curves, one of them closed,
inside a convex, open, smooth domain $\Omega$.
At the initial time the two curves meet
only at one point, the 3-point $O$.
At this point $O$ the curves form angles of $120$ degrees (Herring condition).
The non-closed curve has one fixed end point on the boundary of $\Omega$.
We call such a network a {\em spoon-shaped network}.\\
The evolution is the formal $L^2$ gradient flow of the total length, 
which is the sum of the lengths of the two curves.\\
The first short time existence result 
of a problem similar to ours is due to Bronsard and Reitich;
the existence of a unique solution for the problem 
of the evolution by curvature of three curves, 
meeting at only one point with prescribed angles,
with  Neumann boundary conditions 
of orthogonal intersection with $\partial\Omega$,
was shown in~\cite{bronsard_reitich1993}.
If the planar network, with the Herring condition imposed at triple junctions,
is initially close to the equilibrium configuration, global solutions of the flow 
were obtained by Kinderlehrer and Liu in~\cite{kinderlehrer_liu}.\\
In~\cite{mantegazza_novaga_tortorelli2004, magni_mantegazza_novaga2013,ilmanen_neves_schulze2014} 
it was investigated the long time behavior of the evolution by curvature 
of a triod, that is, a connected network composed by three curves meeting at a 
common point with angles of $120$ degrees and with fixed end points
on the boundary of a convex domain in $\mathbb{R}^2$. 
In~\cite{mantegazza_novaga_tortorelli2004} the authors proved that at the first singular time, either the curvature blows-up
or the length of at least one of the three curves goes to zero.
In~\cite{magni_mantegazza_novaga2013} it is shown that no singularity
can arise during the evolution and this completes the analysis:
either the inferior limit of the length of at least one of the three curves composing the initial triod
goes to zero at some time T,
or the evolution exists for all time and the evolving triod
tends to the Steiner configuration connecting the three fixed points on the boundary.
The authors mention the possibility to generalize several of the results to regular networks.\\
A precise definition of network is given in~\cite{mantegazza_novaga_pluda_schulze2015} 
but the case of network with loops is not taken into account.\\
In this paper, with the same technique as in~\cite{magni_mantegazza_novaga2013}, 
it is studied the simplest example of a regular network with loops.
We call {\em Brakke spoon} a spoon-shaped network
composed by a half-line which meets a closed curve 
shrinking in a self-similar way during the evolution.\\
The main result is the following:
\begin{thm}\label{thmain}
Given a spoon-shaped network $\Gamma_t$ evolving by
curvature in a strictly convex and open set $\Omega\in\mathbb R^2$,
there exists a finite time $T$ such that either the inferior limit of the 
length of the non-closed curve goes to zero as $t\to T$, 
or there exists a point $x_0\in\Omega$ such that
for a subsequence of rescaled time $\mathfrak t_j$
the associate rescaled network 
tends to a Brakke spoon as $j\rightarrow\infty$.
\end{thm}
This result can be seen as an analogue of Grayson's theorem for a closed curve,
which says that an embedded evolving curve in $\mathbb{R}^2$
becomes eventually convex without developing singularities 
and then shrinks to a point (see~\cite{gage_hamilton1986, grayson1987}).
A new proof of this theorem is given in~\cite{magni_mantegazza2014}.\\ 
In Proposition~\ref{resclimit}, we obtain, up to subsequence,
as possible limits, sets that satisfy the equation 
$k+\left\langle x\vert \nu \right\rangle=0$ for all $x$,
where $k$ is the curvature and $\nu$ is the unit normal vector.
Satisfying the previous equation, 
the sets shrink homothetically.
The existence and uniqueness of self-similar shrinking solution
of a problem similar to ours is proved in~\cite{chen_guo}:
Chen and Guo consider an equivalent system that describes the motion by curvature, 
but they focus on the evolution of a curve symmetric about the $x-$axis, and
consider the part of the curve in the upper half-plane,
which forms fixed angles with the axis.
In Proposition~\ref{resclimit}, to gain existence of a unique limit set,
without self-intersections, with multiplicity one,
with curvature not constantly zero, we will apply \cite[Theorem~3]{chen_guo},
as our asymptotic solution is one of the Abresch-Langer curves,
symmetric and convex.\\

\begin{ackn}
The author is warmly grateful to Matteo Novaga for the suggestion of the problem and 
for his constant support and encouragement. 
The author is very grateful to the anonymous referees for their valuable comments, 
which highly improved the paper.\\
The author is partially supported by the
PRIN 2010-2011 project ``Calculus of Variations" 
and by GNAMPA of INDAM.
\end{ackn}

\section{Definitions and preliminaries}
For any $C^2$ curve $\sigma:[0,1]\rightarrow\mathbb{R}^2$ we will denote with
$s$ its arclength parameter, defined by
$s(x)=\int_0^x\vert\sigma_x(\xi)\vert\,d\xi$,
with $\tau=\tau(x)= \frac{\sigma_x }{\vert\sigma_x \vert}$ its unit tangent vector and with
$\nu=\nu(x)= {\mathrm R} \tau(x)={\mathrm R} \frac{\sigma_x}{\vert\sigma_x \vert}$
its unit normal vector, 
where ${\mathrm R}:\mathbb R^2\to\mathbb R^2$ is the counterclockwise rotation of $\pi/2$.\\
The curvature at the point $\sigma(x)$ will be 
$k=k(x)= \frac{\langle\sigma_{xx}\,\vert\,\nu\rangle}
{{\vert{\sigma_x}\vert}^2}=
\langle\partial_s\tau\,\vert\,\nu\rangle=
-\langle\partial_s \nu\,\vert\,\tau\rangle$.\\
Moreover, we set $\underline{k}=k\nu$.\\
A curve $\sigma:[0,1]\rightarrow\mathbb{R}^2$ is called regular if $\sigma_x\not =0$ for all $x\in [0,1]$.

\begin{defn}
Let $\Omega$ be a smooth, open and convex subset of $\mathbb{R}^2$;
a \emph{spoon-shaped network} $\Gamma$ is defined as: 
$\Gamma=\sigma^1([0,1])\cup \sigma^2([0,1]) $,
where $\sigma^1$ and $\sigma^2$ are 
two regular, at least $C^2$ curves contained in $\Omega$:\\ 
$\sigma^1:\left[0,1 \right]\rightarrow \overline {\Omega} \,,$ 
with $\sigma^1(0)=\sigma^1(1)\,,$\\
$\sigma^2:\left[ 0,1 \right]\rightarrow \overline{\Omega}\,.$\\
This network has only one point on the boundary of $\Omega$, $\sigma^2(1)=P\in \partial\Omega$,
which we call the \emph{end point}  of the network.\\
The two curves intersect each other only at one point
$\sigma^1(0)=\sigma^1(1)=\sigma^2(0)=O$,
the \emph{$3-$point} of the network,
moreover $O\notin\partial\Omega$.
At this junction point an angular condition is required, 
namely the unit tangent vectors to the
curves form angles of $120$ degrees,
that is, $\tau^1(0)-\tau^1(1)+\tau^2(0)=0\,.$\\
We require that the network is embedded, that means
\begin{equation*}\label{embedded}
\begin{cases}
\begin{array}{ll}
\sigma^1(x)\ne \sigma^1(y)\qquad &\text{ for}\  x,y\in (0,1), x\neq y\,,\\
\sigma^2(x)\ne \sigma^2(y)\qquad &\text{ for}\  x,y\in (0,1), x\neq y\,,\\
\sigma^1(x)\ne \sigma^2(y) \qquad &\text{ for}\  x,y\in (0,1), x\neq y\,.\\
\end{array}
\end{cases}
\end{equation*}

\end{defn}

\begin{rem}\rm{
If $\Omega$ is not convex, it could happen that during the 
evolution the network intersects the boundary of $\Omega$ in a point 
different from the end point $P$. On the other hand, as a consequence of 
the strong maximum principle for parabolic equations, this cannot happen 
if $\Omega$ is convex.}
\end{rem}

The length of the curves is
$L_i=\int_0^1\vert \sigma^i_x(\xi) \vert d\xi=\int_{\sigma^i}ds\,,$
the global length of the network is $L=L_1+L_2\,$
and we call $A$ the area enclosed in the loop.\\

\noindent The definitions given above of unit tangent vector, unit normal vector and curvature vector
are the same for time depending curves $\gamma(x,t)$.\\
We define a spoon-shaped network $\Gamma_t=\gamma^1(x,t)\cup\gamma^2(x,t)$
as a union of two time depending curves $\gamma^1(x,t)$ and $\gamma^2(x,t)$ above,
with the same properties of  $\sigma^1(x)$ and $\sigma^2(x)$,
for all the times at which they are defined. 

\noindent
For a given initial network $\Gamma_0$, we define the following evolution problem:
\begin{defn} 
We say that a family 
$\Gamma_t=\gamma^1(\cdot, t)\bigcup\gamma^2(\cdot, t)$ 
of spoon-shaped networks evolves
by curvature in the time interval $[0,T)$
if the functions $\gamma^1:[0,1]\times[0,T)\to\overline{\Omega}$ and
$\gamma^2:[0,1]\times[0,T)\to\overline{\Omega}$
are of class $C^2$ in space, $C^1$ in time and satisfy
the following quasilinear parabolic system
\begin{equation}\label{problema}
\begin{cases}
\gamma_x^i(x,t)\not=0\quad &\text{ regularity}\\
\gamma^1(x,t)=\gamma^1(y,t)\quad \text{{ iff} $x=y$ or $x,y\in \{0,1\}$}\\
\gamma^2(x,t)\not=\gamma^2(y,t)\quad \text{{ if} $x\not=y$}&\text{ simplicity}\\
\gamma^1(x,t)=\gamma^2(y,t)\quad \text{{ iff} $x\in \{0,1\},y\in \{0\}$}
&\text{ intersection only  at the 3-point}\\
\tau^1(0)-\tau^1(1)+\tau^2(0)=0 \quad&\text{ angles of $120^\circ$ at the 3-point}\\
\gamma^2(1,t)=P\quad &\text{ fixed end point condition}\\
\gamma^i(x,0)=\sigma^i(x)\quad &\text{ initial data}\\
\gamma^i_t(x,t)=\frac{\gamma_{xx}^i(x,t)}{{\vert{\gamma_x^i(x,t)}\vert}^2}\quad
&\text{ motion by curvature}
\end{cases}
\end{equation}
for every $x\in[0,1]$, $t\in[0,T)$ and $i\in\{1, 2\}$.
\end{defn}

\begin{rem}
\rm{We can set an analogous Neumann problem requiring that the end point
of the non-closed curve
$\gamma^2(1,t)$ intersects orthogonally the boundary of $\Omega$.} 
\end{rem}

Often we will denote the flow with $\Gamma_t$
and we will describe $\Gamma_t$ as a map
$F:\Gamma\rightarrow\overline{\Omega}$ from a fixed standard spoon-shaped network $\Gamma$
in $\mathbb{R}^2$. 
The evolution will be given by a map 
$F:\Gamma\times\left[0,T \right)\rightarrow\overline{\Omega}\,, $
so $\Gamma_t=F\left(\Gamma,t \right)\,.$\\

\begin{rem}
\rm {In the evolution Problem~\eqref{problema} we define
the velocity of the point $\gamma^i(x,t)$ as 
$$\underline{v}^i(x,t)= \frac{\gamma_{xx}^i(x,t)}{{\vert{\gamma_x^i(x,t)}\vert}^2}\,.$$
Notice that $\langle\underline{v}^i,\nu^i\rangle$ is the curvature of the curve 
$\gamma^i$.}
\end{rem}

\begin{lem} 
If $\gamma$ satisfies the equation
\begin{equation}\label{eveq}
\gamma_t=\frac{\gamma_{xx}}{{\vert\gamma_x\vert}^2}=\lambda\tau+k\nu\,,
\end{equation}
then the following commutation rule holds,
\begin{equation}\label{commut}
\partial_t\partial_s=\partial_s\partial_t +
(k^2 -\lambda_s)\partial_s\,.
\end{equation}
\end{lem}

\begin{proof}
The thesis follows by direct computation, 
see~\cite{mantegazza_novaga_tortorelli2004}, Lemma~2.6.
\end{proof}

\begin{lem}
The evolution of the length of the curves is given by 
\begin{align*}
\frac{\partial L_i}{\partial t}
&\,=-\int_{\gamma^i} (k^i)^2-\lambda^i_s\,ds\,,\\
\frac{\partial L}{\partial t}
&=\lambda^1(1,t)-\lambda^1(0,t)-\lambda^2(0,t)-\int_{\Gamma_t} k^2\,ds
\end{align*}
and the evolution of the area enclosed in the loop is
\begin{equation}\label{area}
\frac{\partial A}{\partial t}=-\frac{5\pi}{3}\,.
\end{equation}
\end{lem}

\begin{proof}
The evolution of the length of the curves and of the total length of the
network can be seen as a particular case of a more general situation, studied 
in~\cite{mantegazza_novaga_pluda_schulze2015} or in~\cite{mantegazza_novaga_tortorelli2004}.
The evolution of the area is an application of the Gauss-Bonnet theorem.
\end{proof}

\begin{lem}
At the $3-$point of a smooth spoon-shaped network $\Gamma_t$ evolving as in 
Problem~\eqref{problema}, there hold:\\
\begin{align}\label{cond1}
\lambda^1(0,t)&=\frac{k^{1}(1,t)+k^{2}(0,t)}{\sqrt{3}}\,,\,\nonumber\\
\lambda^1(1,t)&=\frac{k^{2}(0,t)-k^{1}(0,t)}{\sqrt{3}}\,,\,\nonumber\\
\lambda^2(0,t)&=\frac{-k^{1}(0,t)-k^{1}(1,t)}{\sqrt{3}}\,,\nonumber\\
\end{align}
\begin{align}\label{cond2}
k^1(0,t)&=\frac{-\lambda^{1}(1,t)-\lambda^{2}(0,t)}{\sqrt{3}}\,,\,\nonumber\\
k^1(1,t)&=\frac{\lambda^{1}(0,t)-\lambda^{2}(0,t)}{\sqrt{3}}\,,\,\nonumber\\
k^2(0,t)&=\frac{\lambda^{1}(0,t)+\lambda^{1}(1,t)}{\sqrt{3}}\,,\nonumber\\
\end{align}
which imply
\begin{align}\label{cond3}
\lambda^1(0,t)-\lambda^1(1,t)+\lambda^2(0,t)=0\,,\\
k^1(0,t)-k^1(1,t)+k^2(0,t)=0\,.\nonumber
\end{align}
Hence,
\begin{equation*}
\frac{dL}{dt}=-\int_{\Gamma_t}k^2\,ds\,.
\end{equation*}
Moreover
\begin{equation}\label{cond4}
k^1_s(0,t)+\lambda^1(0,t)k^1(0,t)=k^1_s(1,t)+\lambda^1(1,t)k^1(1,t)
=k^2_s(0,t)+ \lambda^2(0,t)k^2(0,t)\,.
\end{equation}
\end{lem}

\begin{proof}
In the system that describes the Problem~\eqref{problema},
at the $3-$point it is required the concurrency condition 
\begin{equation}\label{concurrency}
\gamma^1(1,t)=\gamma^1(0,t)=\gamma^2(0,t)\,.
\end{equation}
If we differentiate in time~\eqref{concurrency},  we obtain 
\begin{align}
\lambda^1(1,t)\tau^1(1,t)+k^1(1,t)\nu^1(1,t)&=\lambda^1(0,t)\tau^1(0,t)+k^1(0,t)\nu^1(0,t)\,,\label{11}\\
\lambda^1(0,t)\tau^1(0,t)+k^1(0,t)\nu^1(0,t)&=\lambda^2(0,t)\tau^2(0,t)+k^2(0,t)\nu^2(0,t)\,,\label{12}\\
\lambda^2(0,t)\tau^2(0,t)+k^2(0,t)\nu^2(0,t)&= \lambda^1(1,t)\tau^1(1,t)+k^1(1,t)\nu^1(1,t)\label{13}\,.
\end{align}
Taking the scalar product of~\eqref{11} by $\tau^2(0,t)$ and by $\lambda^2(0,t)$, 
of~\eqref{12} by $\tau^1(1,t)$ and by $\lambda^1(1,t)$, of~\eqref{13} by $\tau^1(0,t)$ and $\lambda^1(0,t)$
and using the angular condition $\tau^1(0,t)-\tau^1(1,t)+\tau^2(0,t)=0$
we get 
\begin{align*}
-\frac{1}{2}\lambda^1(0,t)-\frac{\sqrt{3}}{2}k^1(0,t)&=\frac{1}{2}\lambda^1(1,t)-\frac{\sqrt{3}}{2}k^1(1,t)\,,\nonumber\\
\frac{\sqrt{3}}{2}\lambda^1(0,t)-\frac{1}{2}k^1(0,t)&=\frac{\sqrt{3}}{2}\lambda^1(1,t)+\frac{1}{2}k^1(1,t)\,,\nonumber\\
\frac{1}{2}\lambda^1(1,t)+\frac{\sqrt{3}}{2}k^1(1,t)&=-\frac{1}{2}\lambda^2(0,t)+\frac{\sqrt{3}}{2}k^2(0,t)\,,\nonumber\\
-\frac{\sqrt{3}}{2}\lambda^1(1,t)+\frac{1}{2}k^1(1,t)&=-\frac{\sqrt{3}}{2}\lambda^2(0,t)-\frac{1}{2}k^2(0,t)\,,\nonumber\\
\frac{1}{2}\lambda^2(0,t)+\frac{\sqrt{3}}{2}k^2(0,t)&=\frac{1}{2}\lambda^1(0,t)-\frac{\sqrt{3}}{2}k^1(0,t)\,,\nonumber\\
-\frac{\sqrt{3}}{2}\lambda^2(0,t)+\frac{1}{2}k^2(0,t)&=-\frac{\sqrt{3}}{2}\lambda^1(0,t)+\frac{1}{2}k^1(0,t)\,,\nonumber\\
\end{align*}
from which we obtain~\eqref{cond1} and~\eqref{cond2}.
Summing up the previous conditions we have~\eqref{cond3}.
Differentiating in time the angular condition at the $3-$point
$\tau^1(0,t)-\tau^1(1,t)+\tau^2(0,t)=0$,
using the commutation rule~\ref{commut},
which implies $\partial_t\tau=\left(k_s+k\lambda\right)\nu$,
we have~\eqref{cond4}.
\end{proof}

\subsection{Short time existence}
\begin{defn}
We say that for Problem~\eqref{problema} the {\em compatibility conditions} (of every order) are satisfied if
at the end--point and at the 3--point there hold all the relations on the space derivatives 
of the functions $\sigma^i$ obtained differentiating (any number of times) in time the  boundary conditions
and substitute them with the space derivatives, given by using the evolution equation~\eqref{eveq}. 
\end{defn}

\begin{rem}\label{compc} \rm{
Explicitly, for Problem~\eqref{problema}, the compatibility conditions 
of order $2$ are
\begin{equation*}
\frac{\sigma_{xx}^2(1)}{|\sigma_x^2(1)|^2}=0 \,,
\end{equation*}
and
\begin{equation*}
\frac{\sigma_{xx}^1(0)}{|\sigma_x^1(0)|^2}=\frac{\sigma_{xx}^1(1)}{|\sigma_x^1(1)|^2}\;,
\frac{\sigma_{xx}^1(0)}{|\sigma_x^1(0)|^2}=\frac{\sigma_{xx}^2(0)}{|\sigma_x^2(0)|^2}\;,
\frac{\sigma_{xx}^2(0)}{|\sigma_x^2(0)|^2}=\frac{\sigma_{xx}^1(1)}{|\sigma_x^1(1)|^2}\;.
\end{equation*}}
\end{rem}

\smallskip 

\begin{thm}
If  $\Gamma_0$ is an initial $C^{2+2\alpha}$ network
(with $0<\alpha<1/2$) satisfying the compatibility conditions of order $2$,
then, there exists a unique solution 
$\Gamma \in C^{2+2\alpha,1+\alpha}\left(\left[0,1 \right]\times \left[0,T\right) \right)$
of Problem~\eqref{problema}.
\end{thm}

\begin{proof}
The proof is based on the results of~\cite{bronsard_reitich1993}.
In that paper they consider three curves lying in a convex and smooth 
domain $\Omega\subset\mathbb{R}^2$, intersecting at a point with prescribed angle
and meeting the boundary of $\Omega$ orthogonally.
Linearizing the problem about the initial data
and using a fixed point argument, Bronsard and Reitich obtain a
$C^{2+2\alpha,1+\alpha}\left(\left[0,1 \right]\times \left[0,T\right) \right)$ solution.
The proof can be easily adapted to our situation.
\end{proof}

\begin{thm}
For any initial smooth, regular network $\Gamma_0$ in a smooth, convex, open set  
$\Omega\subset\mathbb{R}^2$ there exists a unique smooth solution of 
Problem~\eqref{problema} in a maximal time interval $\left[ 0,T\right) $.
\end{thm}

\begin{proof}
See~\cite{mantegazza_novaga_tortorelli2004}, Theorem~3.1.
\end{proof}

\subsection{Integral estimates}
Following the line of proof of~\cite{mantegazza_novaga_tortorelli2004}, we use 
integral estimates to describe the behavior of the evolution of the network.
Thanks to the conditions~\eqref{cond1}, \eqref{cond2},
\eqref{cond3} and \eqref{cond4} at the $3-$point 
and to the Dirichlet condition at the end point $P$,
the computation in~\cite{mantegazza_novaga_tortorelli2004}, Section 3, pp.257-263,
can be repeated, to obtain:
$$
\int_{\Gamma_t} |\partial_s^j k|^2\,ds
\leq C\int_{0}^t\left(\int_{\Gamma_\xi}
 k^2\,ds\right)^{2j+3}\,d\xi + C\left(\int_{\Gamma_t}
  k^2\,ds\right)^{2j+1} + Ct + C\,,
$$
where with $\partial^j_s k^2$ we denote the $j$-th derivative of $k^2$ 
along a curve with respect to the relative arclength parameter and $C$ is a 
constant depending only on $j\in\mathbb N$ and the initial network $\Gamma_0$.
Passing from integral to $L^\infty$ estimates by using Gagliardo-Nirenberg 
interpolation inequalities (see~\cite{nirenberg59}, equation $(2.2)$, page $125$), we have the following:

\begin{prop}\label{pluto1000} If the lengths of the curves are
  positively bounded from below and the $L^2$ norm of
  $k$ is bounded, uniformly on $[0,T)$, then the curvature of $\Gamma_t$ 
  and all its space derivatives are uniformly bounded in the same
  time interval by some constants depending only on the $L^2$ integrals of
  the space derivatives of $k$ on the initial network $\Gamma_0$.
\end{prop}

\begin{rem}\label{0001}
\rm{If the length $L_1$ of the closed curve $\gamma^1$ 
is not bounded from below, 
then the $L^2$ norm of its curvature $k$ is not bounded.
In fact,
from previous estimates and using H\"older inequality, we get
$\frac{5\pi}{3}=\vert\int_{\gamma^1} k ds\vert\leq \left( L_1\right)^{\frac12}\left(\int_{\gamma^1} k^2 ds\right)^{\frac12}$, hence
\begin{equation}\label{curvlen}
\int_{\gamma^1} k^2ds\geq \frac{25\pi^2}{9L_1}\,.
\end{equation}
Moreover the maximal time $T$ of existence of a smooth flow is finite and
\begin{equation}\label{stimaT}
T\leq \frac{3A_0}{5\pi}\,,
\end{equation}
because, in order to avoid the presence of a singularity, the area $A(t)$
has to stay positive during the smooth evolution, otherwise the $L^2$ norm of the 
curvature is not bounded.
From \eqref{area} we easily obtain \eqref{stimaT}.
}
\end{rem}

\noindent
The next theorem follows from Propositions~\ref{pluto1000} and~\ref{0001} 
reasoning as in~\cite[Theorem~3.18]{mantegazza_novaga_tortorelli2004}.

\begin{thm}\label{infinitecurvature}
If $[0,T)$, with $T<+\infty$, is the maximal time interval of existence of a smooth solution $\Gamma_t$ of Problem~\eqref{problema}, 
then at least one of the following possibilities holds:
\begin{align*}
\bullet &\; \liminf_{t\rightarrow T}\, L_2(t) = 0\,;\\
\bullet &\; \limsup_{t\rightarrow T} \int_{\Gamma_t}k^2 ds = +\infty\,.
\end{align*}
\end{thm}

\noindent
In the whole paper we will require that the length of the non-closed curve $\gamma^2$ 
is positively bounded from below during the flow, 
to avoid that the non-closed curve shrinks to a point faster than the closed one.

\section{An isoperimetric estimate}
Given the smooth  flow $\Gamma_t=F\left(\Gamma,t\right)$ 
we take two points $p=F\left(x,t\right)$ and $q=F\left(y,t\right)$ 
belonging to $\Gamma_t$ and we call $\left( \xi_{p,q}\right)_i$ 
the geodesic curves contained in 
$\Gamma_t$ connecting $p$ and $q$. \\
Then, if $p$ and $q$ are both points of the closed curve $\gamma^1$,
we let $A$ be the area of the region $\mathcal{A}$ in $\mathbb{R}^2$
enclosed in the loop
and  $A_{i,p,q}$ the area of the open region $\mathcal{A}_{i,p,q}$ enclosed by the segment
$[p,q]$ and one of the geodesic curve $\left(\xi_{p,q}\right)_i$. (See Fig.~\ref{disuelle}).\\
When the ${\mathcal A}_{i,p,q}$ is not connected, we let
$A_{i,p,q}$ to be the sum of the areas of its connected components.
We define $\psi(A_{p,q})$ as 
$$
\psi(A_{p,q})=\frac{A}{\pi}\sin\left( \frac{\pi}{A}A_i\right)\,,
$$
where $A_i$ is the smallest of possible $A_{i,p,q}$,
depending on the choice of $\left(\xi_{p,q}\right)_i$.\\ 
 If $p$ and $q$ are not both points of $\gamma^1$,
 we let again $A_{i,p,q}$ be the area of the open region $\mathcal{A}_{i,p,q}$ enclosed by the segment
$[p,q]$ and one of the geodesic curves $\left(\xi_{p,q}\right)_i$.
Like before, when the region ${\mathcal A}_{i,p,q}$ is not connected, we let
$A_{i,p,q}$ be the sum of the areas of its connected components.
Differently from the previous case, the possible ${\mathcal A}_{i,p,q}$
are not disjoint.
We call $A_{p,q}$  the area of the smallest one
(See Fig.~\ref{disuelle}).

We consider the function $\Phi_t:\Gamma\times\Gamma\rightarrow\mathbb{R}\cup{\{+\infty\}}$
defined as
\begin{equation*}
\Phi_t(x,y)= 
\begin{cases}
\frac{\vert p-q\vert^2}{\psi(A_{p,q})}\qquad &\,\text{ {if} $x\not=y$},\,x,y\; 
\text{are points of $\gamma^1$},\\
\frac{\vert p-q\vert^2}{A_{p,q}}\qquad &\text{ { if} $x\not=y$},\,x,y\;\text{are 
not both points of $\gamma^1$},\\
4\sqrt{3}\;\; &\text{ {  if} $x$ and $y$ coincide with the $3-$point $O$ of $\Gamma$}\,,\\
+\infty\;\; &\text{ {  if} $x=y\not=O$},\\
\end{cases}
\end{equation*} 
where $p=F(x,t)$ and $q=F(y,t)$.\\
Since $\Gamma_t$ is smooth and the $120$ degrees condition holds, it is
easy to check that $\Phi_t$ is a lower semicontinuous function. Hence,
by the compactness of $\Gamma$, the following infimum is actually a
minimum
\begin{equation*}
E(t) = \inf_{x,y\in\Gamma}\Phi_t(x,y)\,,
\end{equation*}
for every $t\in[0,T)$.
We call $E(t)$ ``embeddedness measure''.
Similar geometric quantities have already been applied to similar problems 
in~\cite{hamilton3},~\cite{chzh} and~\cite{huisk2}.\\
The function $\Phi_t$ is locally Lipschitz because the numerator is the square of a distance function
and until the flow exists the denominator never goes to zero.

\begin{rem}\rm{Following the proof of Theorem~2.1 in~\cite{chzh}, 
we can always find a minimizing pair $\left(p,q\right)$ such that $A_{i}$
is composed only by one connected component.
}\end{rem}

\begin{rem}\rm{
Following the line of proof by Huisken in~\cite{huisk2},
in the definition of the function $\Phi_t$ we distinguish two cases. 
We introduce the function $\psi(A_{p,q})$ when
the two points belong to the closed curve $\gamma^1$, because we want that $\Phi_t$ is a smooth
function also when $A_i$ is equal to $\frac{A}{2}$.
As we can exclude the case that $A_i$ has more that one connected component, as we said in the previous remark,
it can happen that $A_1=A_2$ only when both $p$ and $q$ belong to the closed curve $\gamma^1$.}
\end{rem}

In the following computation we only consider the case in which
$p$ and $q$ are both points of the closed curve $\gamma^1$; for the other situation,
see Section $4$ of~\cite{mantegazza_novaga_tortorelli2004}.\\

\begin{center}
\begin{tikzpicture}[scale=1]
\draw[color=black,scale=1,domain=-4.141: 4.141,
smooth,variable=\t,shift={(-1,0)},rotate=0]plot({3.25*sin(\t r)},
{2.5*cos(\t r)}) ;
\draw (-4.25,0) node[left] {$P$} to[out=30,in=110, looseness=1] (-3.03,0) 
to[out= -50,in=180, looseness=1] (-1.81,0)
to[out=60, in=35, looseness=1] (-1, 0.81)
to[out=35,in=150,  looseness=1] (0.62,0.81)
to[out=-30,in=90, looseness=1.5] (1.44,0)
to[out=-90,in=20, looseness=1] (0.62,-0.81) 
to[out=-160,in=-60,  looseness=0.5] (-1.81,0);
\draw (-3.03,0) coordinate (a_1) -- (-1,0.81) coordinate (a_2);
\filldraw[fill=black!10!white]
(-3.03,0) 
to[out= -50,in=180, looseness=1] (-1.81,0) -- (-1.81,0)
to[out=60, in=35, looseness=1] (-1, 0.81) --
(-3.03,0) coordinate (a_1);
\path[font= \scriptsize] 
(-2.3,0.3) node[below] {$\mathcal{A}_{p,q}$};
\path[font= \Large]
(-3.75,-1.8) node[below] {$\Omega$};
\path[font=\large]
    (-3.8,0.5) node[right] {$\gamma^2$}
    (0.5,-1) node[right] {$\gamma^1$}
    (-1.9,0) node[below] {$O$}
    (-3.03,0) node[below] {$p$}
    (-1,0.81) node[below] {$q$};
\draw[color=black,scale=1,domain=-4.141: 4.141,
smooth,variable=\t,shift={(7,0)},rotate=0]plot({3.25*sin(\t r)},
{2.5*cos(\t r)}) ;
\draw (3.75,0) node[left] {$P$} to[out=30,in=110, looseness=1] (4.97,0) 
to[out= -50,in=180, looseness=1] (6.19,0)
to[out=60, in=35, looseness=1] (7, 0.81)
to[out=35,in=150,  looseness=1] (8.62,0.81)
to[out=-30,in=90, looseness=1.5] (9.44,0)
to[out=-90,in=20, looseness=1] (8.62,-0.81) 
to[out=-160,in=-60,  looseness=0.5] (6.19,0);
\draw (7,0.81) coordinate (a_1) -- (9.44,0) coordinate (a_2);
\filldraw[fill=black!10!white]
(7,0.81) 
to[out=60,in=150,  looseness=1] (8.62,0.81)-- (8.62,0.81)
to[out=-30,in=90, looseness=1.5] (9.44,0) -- (7,0.81) coordinate (a_1);
\path[font= \scriptsize] 
(8,1.1) node[below] {$\mathcal{A}_{1,p,q}$}
(7.7,0.3) node[below]{$\mathcal{A}_{2,p,q}$};
\path[font= \Large]
(4.25,-1.8) node[below] {$\Omega$};
\path[font=\large]
    (4.2,0.5) node[right] {$\gamma^2$}
    (8.5,-1) node[right] {$\gamma^1$}
    (6.1,0) node[below] {$O$}
    (7,0.81) node[below] {$p$}
    (9.44,0) node[right] {$q$};                 
\end{tikzpicture}
  \end{center}
\begin{fig}\label{disuelle}
In the figure are shown two possible situations: in the first case the point $p\in\gamma^2$ and 
$q\in\gamma^1$, in the second case $p$ and $q$ are both point of $\gamma^1$.
\end{fig}

If the network $\Gamma_t$ has no self-intersections we have
$E(t)>0$; the converse is clearly also true.\\
Moreover, $E(t)\leq\Phi_t(0,0)=4\sqrt{3}$ always holds, thus when
$E(t)>0$ the two points $(p,q)$ of a minimizing pair $(x,y)$ can coincide if
and only if $p=q=O$.\\
Finally, since the evolution is smooth, it is easy to see that the
function $E:[0,T)\to\mathbb R$ is continuous.\\

\noindent
Since we are dealing with embedded networks, with a little abuse of
notation, we consider the function $\Phi_t$ defined on
$\Gamma_t\times\Gamma_t$ and we speak of a minimizing pair for the couple
of points $(p,q)\in\Gamma_t\times\Gamma_t$ instead of
$(x,y)\in\Gamma\times\Gamma$.

\begin{prop}\label{lemet2} The function $E(t)$ is monotone increasing 
in every time interval where $0<E(t)<\frac{1}{2}$.
\end{prop}
\begin{proof} 
We assume that $0<E(t)<\frac{1}{2}$. 
Since $E(t)$ is a locally Lipschitz function,
to prove the statement it is then enough to show that
$\frac{dE(t)}{dt}>0$ for every time $t$ at which this derivative
exists (which happens almost everywhere in every time interval where $0<E(t)<\frac{1}{2}$).\\
We can exclude the case in which $p$ or $q$ coincides with $O$,
the proof goes like the one of Lemma~4.2, in~\cite{mantegazza_novaga_tortorelli2004}.\\
Fixed a minimizing pair $(p,q)$ at time $t$,  
we choose a value  $\varepsilon>0$ smaller than the intrinsic distances of $p$ and $q$ from
the 3--point $O$ of $\Gamma_t$ and between them.\\
Possibly taking a smaller $\varepsilon>0$, we fix an arclength
coordinate $s\in (-\varepsilon,\varepsilon)$ and a local parametrization $p(s)$ of
the curve containing in a neighborhood of $p=p(0)$, with the 
same orientation of the original one.
Let $\eta(s)=\vert p(s)-q\vert$ and $A_i(s)=A_{i,p(s),q}$, since  
\begin{equation*}
E(t)=\min_{s\in(-\varepsilon,\varepsilon)}\frac{\eta^2(s)}{\psi(A_{p(s),q})}=
\frac{\eta^2(0)}{\psi(A_{p(s),q}(0))}\,,
\end{equation*}
if we differentiate in $s$ we obtain 
\begin{equation}\label{eqdsul}
\frac{d\eta^2(0)}{ds}\psi(A_{p(s),q}(0))=
\frac{d\psi(A_{p(s),q}(0))}{ds}\eta^2(0)\,.
\end{equation}
As the intersection of the segment $[p,q]$ with the network is transversal, 
we have an angle $\alpha(p)\in(0,\pi)$ 
determined by the unit tangent $\tau(p)$ and the vector $q-p$.\\
We compute
\begin{align*} 
\frac{{ d} \eta^2(0)}{{ d} s} &=\,
-2 \langle \tau(p)\,\vert\, q-p\rangle = -2 \vert
p-q\vert \cos\alpha(p)\\ 
\frac{{ d} A(0)}{{ d} s} &=\, 0\\
\frac{{ d} A_i(0)}{{ d} s} &=\, 
\frac 12 \vert \tau(p) \wedge (q-p)\vert = 
\frac 12 \langle \nu(p)\,\vert\, q-p\rangle = 
\frac 12  \vert p-q\vert \sin\alpha(p)\\
\frac{{ d} \psi(A_{p(s),q}(0))}{{ d} s} &=\, 
\frac{{ d} A_i(0)}{{ d} s}\cos\left(\frac{\pi}{A}A_i(0)\right)\\
\, &=\,
\frac 12  \vert p-q\vert \sin\alpha(p)\cos\left(\frac{\pi}{A}A_i(0)\right)\,. \\
\end{align*}
Putting these derivatives in equation~\eqref{eqdsul} and
recalling that $\eta^2(0)/\psi(A_{p(s),q}(0))=E(t)$,  we get
\begin{align}\label{topolino2}
\cot\alpha(p) 
&\,= -\frac{\vert p-q\vert^2}{4\psi(A_{p(s),q}(0))}\cos\left(\frac{\pi}{A}A_i(0)\right)\nonumber\\
&\,= -\frac{E(t)}{4}\cos\left(\frac{\pi}{A}A_i(0)\right)\,.
\end{align}
Since $0<E(t)<\frac{1}{2}<\frac{4}{\sqrt{3}}$ we get $-\frac{1}{\sqrt{3}}<\cot\alpha(p)<0$ which
implies
\begin{equation*}
\frac{\pi}{2} <\alpha(p) < \frac 23 \pi\,.
\end{equation*}
The same argument clearly holds for the point $q$, hence 
defining $\alpha(q)\in(0,\pi)$ to be the angle determined by the
unit tangent $\tau(q)$ and the vector $p-q$, by equation~\eqref{topolino2}
it follows that $\alpha(p)=\alpha(q)$ and we simply write $\alpha$
for both.\\
We consider now a different variation, moving at the same time the
points $p$ and $q$, in such a way that $\frac{dp(s)}{ds}=\tau(p(s))$ and 
$\frac{dq(s)}{ds}=\tau(q(s))$.\\
As above, letting $\eta(s)=\vert p(s)-q(s)\vert$ and 
$A(s)=A_{i\,(p(s),q(s))}$, by minimality we have
\begin{align}\label{eqdersec}
\frac{{ d} \eta^2(0)}{{ d} s}\psi(A_{p(s),q(s)}(0))&=
\frac{{ d} \psi(A_{p(s),q(s)}(0))}{{ d} s}\eta^2(0) \;\;\text{ { and} }\;\;\nonumber\\
\frac{{ d}^2 \eta^2(0)}{{ d} s^2}\psi(A_{p(s),q(s)}(0))&\ge 
\frac{{ d}^2\psi( A_{p(s),q(s)}(0))}{{ d} s^2}\eta^2(0)\,.
\end{align}
Computing as before,
\begin{align*}
\frac{{ d} \eta^2(0)}{{ d} s} 
&=\,2 \langle p-q \,\vert\,\tau(p)-\tau(q)\rangle = -4\vert p-q\vert\cos\alpha \\
\frac{{ d} A(0)}{{ d} s} 
&=\, 0\\
\frac{{ d} A_i(0)}{{ d} s} 
&=\, -\frac12\langle p-q\,\vert\,\nu(p)+\nu(q)\rangle
=+\vert p-q\vert\sin\alpha\\
\frac{{ d}^2 \eta^2(0)}{{ d} s^2} 
&=\,2 \langle \tau(p)-\tau(q) \,\vert\,\tau(p)-\tau(q)\rangle +
2 \langle p-q \,\vert\, k(p)\nu(p) -k(q)\nu(q)\rangle\\
&=\, 2\vert \tau(p)-\tau(q)\vert^2+
2 \langle p-q \,\vert\, k(p)\nu(p) -k(q)\nu(q)\rangle\\
&=\, 8\cos^2\alpha+
2 \langle p-q \,\vert\, k(p)\nu(p) -k(q)\nu(q)\rangle\\
\frac{{ d}^2 A_i(0)}{{ d} s^2} 
&=\,-\frac12\langle \tau(p)-\tau(q)\,\vert\,\nu(p)+\nu(q)\rangle
+\frac12\langle p-q\,\vert\,k(p)\tau(p)+k(q)\tau(q)\rangle\\
&=\,-\frac12\langle\tau(p)\,\vert\,\nu(q)\rangle
+\frac12\langle \tau(q)\,\vert\,\nu(p)\rangle\\
&+\frac12\langle p-q\,\vert\,k(p)\tau(p)+k(q)\tau(q)\rangle\\
&=\,-2\sin\alpha\cos\alpha
-1/2\vert p-q\vert(k(p)-k(q))\cos\alpha\\
\frac{{ d}^2 \psi(A_{p(s),q(s)}(0))}{{ d} s^2} 
&=\,\left.\frac{{ d}}{{ d} s}\left\lbrace
\frac{{ d} A_i(s)}{{ d} s}
\cos\left(\frac{\pi}{A}A_i(s)\right) 
\right\rbrace\right\vert_{s=0}\\
&=\, \frac{{ d}^2 A_i(0)}{{ d} s^2}\cos\left(\frac{\pi}{A}A_i(0)\right)\\
&\, -\frac{\pi}{A}\left(\frac{{ d} A_i(0)}{{ d} s}\right)^2\sin\left(\frac{\pi}{A}A_i(0)\right) \\
&=\, (-2\sin\alpha\cos\alpha-\frac12\vert p-q\vert(k(p)-k(q))\cos\alpha)
\cos\left( \frac{\pi}{A}A_i(0)\right) \\
&\,-\frac{\pi}{A}\vert p-q\vert^2 \sin^2\alpha\sin\left(\frac{\pi}{A}A_i(0)\right)\,. \\
\end{align*} 
Substituting the last two relations in the second inequality 
of~\eqref{eqdersec}, we get 
\begin{align*}
&(8\cos^2\alpha+\, 
2\langle p-q \,\vert\, k(p)\nu(p) -k(q)\nu(q)\rangle)\psi(A_{p(s),q(s)}(0))\\
&\geq\, \vert p-q\vert^2
\left\lbrace 
(-2\sin\alpha\cos\alpha-\frac12\vert p-q\vert(k(p)-k(q))\cos\alpha)
\cos\left( \frac{\pi}{A}A_i(0)\right)\right.\\
&\,\left.-\frac{\pi}{A}\vert p-q\vert^2 \sin^2\alpha\sin\left(\frac{\pi}{A}A_i(0)\right) 
\right\rbrace \,,
\end{align*}
hence, keeping in mind that 
$\tan\alpha=
\frac{-4}{E(t)\cos\left(\frac{\pi}{A}A_i \right) }$, we obtain
\begin{align}\label{eqfin}
&2\psi(A_{p(s),q(s)}(0))\langle p-q \,\vert\, \,k(p)\nu(p) -k(q)\nu(q)\rangle\nonumber\\
&\,+1/2\vert p-q\vert^3(k(p)-k(q))\cos\alpha 
\cos\left(\frac{\pi}{A}A_i(0)\right)\nonumber\\
&\,\geq -2\sin\alpha\cos\alpha\vert p-q\vert^2\cos\left(\frac{\pi}{A}A_i(0)\right)\nonumber\\
&\,-8\psi(A_{p(s),q(s)}(0))\cos^2\alpha
+\vert p-q\vert^4 \sin^2\alpha
\left[-\frac{\pi}{A}\sin\left(\frac{\pi}{A}A_i(0)\right) \right]\nonumber\\
&=\,-2\psi(A_{p(s),q(s)}(0))\cos^2\alpha
\left(\tan\alpha\frac{\vert p-q\vert^2}{\psi(A_{p(s),q(s)})}
\cos\left(\frac{\pi}{A}A_i(0)\right)+ 4\right)\nonumber\\
&\,+\vert p-q\vert^4\sin^2\alpha
\left[ -\frac{\pi}{A}\sin\left(\frac{\pi}{A}A_i(0)\right) \right] \nonumber\\
&=\,+\vert p-q\vert^4\sin^2\alpha
\left[ -\frac{\pi}{A}\sin\left(\frac{\pi}{A}A_i(0)\right) \right]\,. \nonumber\\
\end{align}
We consider now a time $t_0$ such that the derivative
$\frac{dE(t_0)}{dt}$ exists and we compute it with the Hamilton's trick (see \cite{hamilton}),
that is,
$$
\frac{dE(t_0)}{dt} = \left.\frac{\partial}{\partial
    t}\Phi_t(p,q)\right\vert_{t=t_0}
$$
for {\em any} pair $(p,q)$ such that $p, q\in\Gamma_{t_0}$ and $\frac{\vert
  p-q\vert^2}{\psi(A_{p,q})}=E(t_0)$.\\
Considering then a minimizing pair $(p,q)$ for $\Phi_{t_0}$ with
all the previous properties, by minimality, we are free to choose the 
``motion'' of the points $p(t)$, $q(t)$ ``inside'' the networks $\Gamma_t$
in computing such partial derivative.\\
Since locally the networks are moving by curvature and we know that
neither  $p$ nor $q$ coincides with the 3--point, 
we can find $\varepsilon>0$ and two  smooth curves $p(t), q(t)\in\Gamma_t$  for
every $t\in(t_0-\varepsilon,t_0+\varepsilon)$ such that
\begin{align*}
p(t_0) &=\, p \qquad \text{ and }\qquad \frac{dp(t)}{dt} = k(p(t), t)
~\nu(p(t),t)\,, \\
q(t_0) &=\, q \qquad \text{ and }\qquad \frac{dq(t)}{dt}= k(q(t),t)
~\nu(q(t),t)\,. 
\end{align*}
Then,
\begin{align}\label{eqderE}
\frac{dE(t_0)}{dt} =& 
\left.\frac{\partial}{\partial t}\Phi_{t}(p,q)\right\vert_{t=t_0}\nonumber\\
=&\frac{1}{(\psi(A_{p,q}(t)))^2}\left.
\left(\psi(A_{p,q}(t))
\frac{d\vert p(t)-q(t)\vert^2}{dt} - 
\vert p-q\vert^2 \frac{d\psi(A_{p,q}(t))}{dt}\right)\right\vert_{t=t_0}\,.
\end{align}
With a straightforward computation we get the following equalities, 
\begin{align*}
\left.\frac{{ d} \vert p(t)-q(t)\vert^2}{{ d} t} \right\vert_{t=t_0}
&\,= 2 \langle p-q\,\vert\, k(p)\nu(p) -k(q)\nu(q)\rangle\\
\left.\frac{{ d} (A(t))}{{ d} t}\right\vert_{t=t_0} 
&\,= -\frac{5\pi}{3}\\
\left.\frac{{ d} A_i(t)}{{ d} t}\right\vert_{t=t_0} 
&\,=\int_{\Gamma_{p,q}} \langle\underline{k}(s)\,\vert\nu_{\xi_{p,q}}\rangle\,ds
+ \frac12\vert p-q\vert\langle {\nu_{[p,q]}}\,\vert\, k(p)\nu(p)+k(q)\nu(q)\rangle\\
&\,= 2\alpha -\frac{5\pi}{3}-\frac 12 \vert p-q\vert(k(p)-k(q))\cos\alpha\\
\end{align*}
\begin{align*}
\left.\frac{{ d} \psi(A_{p,q}(t))}{{ d} t} \right\vert_{t=t_0}
&\,=\left.\frac{{ d} (A(t))}{{ d} t}\right\vert_{t=t_0}
\left[\frac{1}{\pi}\sin\left(\frac{\pi}{A(t_0)}A_i(t_0)\right)\right] \\
&\,+\cos\left(\frac{\pi}{A(t_0)}A_i(t_0) \right)
\left(\left.\frac{{ d} (A_i(t))}{{ d} t}\right\vert_{t=t_0}-\frac{A_i(t_0)}{A(t_0)}
\left.\frac{{ d} (A(t))}{{ d} t}\right\vert_{t=t_0} \right) \\
&\,=\frac{{ d} (A(t_0))}{{ d} t}
\left[\frac{1}{\pi}\sin\left(\frac{\pi}{A(t_0)}A_i(t_0)\right)
-\frac{A_i(t_0)}{A(t_0)}\cos\left(\frac{\pi}{A(t_0)}A_i(t_0) \right)\right] \\
&\,+\frac{{ d} (A_i(t_0))}{{ d} t}\cos\left(\frac{\pi}{A(t_0)}A_i(t_0) \right)\\
&\,=-\frac{5\pi}{3}\left[\frac{1}{\pi}\sin\left(\frac{\pi}{A(t_0)}A_i(t_0)\right)
-\frac{A_i(t_0)}{A(t_0)}\cos\left(\frac{\pi}{A(t_0)}A_i(t_0) \right)\right] \\
&\,+\left(2\alpha -\frac{5\pi}{3}-\frac 12 \vert p-q\vert(k(p)-k(q))\cos\alpha\right)
\cos\left(\frac{\pi}{A(t_0)}A_i(t_0) \right)\\
\end{align*}
where we wrote $\nu_{\xi_{p,q}}$ and $\nu_{[p,q]}$ for the exterior
unit normals to the region $A_i$, respectively at the
points of the geodesic $\xi_{p,q}$ and of the segment $[p,q]$.\\
Substituting these derivatives in equation~\eqref{eqderE} we get
\begin{align*}
\frac{dE(t_0)}{dt} 
=&\,\frac{2\langle p-q\,\vert\, k(p)\nu(p) -k(q)\nu(q)\rangle}{\psi(A_{p,q}(t_0))}\\
-&\frac{\vert p-q\vert^2}{(\psi(A_{p,q}(t_0)))^2}
\left\lbrace 
-\frac{5\pi}{3}\left[\frac{1}{\pi}\sin\left(\frac{\pi}{A(t_0)}A_i(t_0)\right)
-\frac{A_i(t_0)}{A(t_0)}\cos\left(\frac{\pi}{A(t_0)}A_i(t_0) \right)\right]
\right. \\
&\left. +\left(2\alpha -\frac{5\pi}{3}-\frac 12 \vert p-q\vert(k(p)-k(q))\cos\alpha\right)
\cos\left(\frac{\pi}{A(t_0)}A_i(t_0) \right)
 \right\rbrace \\
\end{align*}
and, by equation~\eqref{eqfin}, \\
\begin{align*}
\frac{dE(t_0)}{dt} 
\geq 
&\,-\frac{\vert p-q\vert^2}{(\psi(A_{p,q}))^2}
\left\lbrace 
-\frac{5}{3}\sin\left(\frac{\pi}{A}A_i\right)
+\frac{5\pi}{3}\frac{A_i}{A}\cos\left(\frac{\pi}{A}A_i \right) \right.\\
&\,\left. +\left(2\alpha -\frac{5\pi}{3}\right)
\cos\left(\frac{\pi}{A}A_i \right)+\frac{\pi}{A}\vert p-q \vert^2 \sin^2(\alpha) 
\sin\left(\frac{\pi}{A}A_i\right)
\right\rbrace\,.\\
\end{align*}

It remains to prove that the quantity 
\begin{align*}
&\frac{5}{3}\sin\left(\frac{\pi}{A}A_i\right)
-\frac{5\pi}{3}\frac{A_i}{A}\cos\left(\frac{\pi}{A}A_i \right) 
 +\left(\frac{5\pi}{3}-2\alpha\right)
\cos\left(\frac{\pi}{A}A_i \right)\\
&-\frac{\pi}{A}\vert p-q \vert^2 \sin^2(\alpha) 
\sin\left(\frac{\pi}{A}A_i\right)\\
&= \frac{5}{3}\sin\left(\frac{\pi}{A}A_i\right)
-\frac{5\pi}{3}\frac{A_i}{A}\cos\left(\frac{\pi}{A}A_i \right) 
 +\left(\frac{5\pi}{3}-2\alpha\right)
\cos\left(\frac{\pi}{A}A_i \right)\\
&-E(t)\sin^2(\alpha) 
\sin^2\left(\frac{\pi}{A}A_i\right)
\end{align*}
is positive.\\
If $0\leq \frac{A_i}{A}\leq \frac 13$,
we have
\begin{align*}
\frac{dE(t_0)}{dt} 
&\geq 
\,  
\frac{5}{3}\sin\left(\frac{\pi}{A}A_i\right)
-\frac{5\pi}{3}\frac{A_i}{A}\cos\left(\frac{\pi}{A}A_i \right) \\
& +\left(\frac{5\pi}{3}-2\alpha\right)
\cos\left(\frac{\pi}{A}A_i \right)-E(t)\sin^2(\alpha) 
\sin^2\left(\frac{\pi}{A}A_i\right)\\
&\, \geq 
\left(\frac{5\pi}{3}-2\alpha\right)
\cos\left(\frac{\pi}{A}A_i \right)-E(t)\sin^2(\alpha) 
\sin^2\left(\frac{\pi}{A}A_i\right)\\
&\, \geq
\left(\frac{\pi}{3}\right)
\cos\left(\frac{\pi}{3} \right)-E(t) 
\sin^2\left(\frac{\pi}{3}\right)>0\,.\\
\end{align*}
If $\frac 13\leq \frac{A_i}{A}\leq \frac 12$,
we get
\begin{align*}
\frac{dE(t_0)}{dt} 
&\geq 
\,  
\frac{5}{3}\sin\left(\frac{\pi}{A}A_i\right)
-\frac{5\pi}{3}\frac{A_i}{A}\cos\left(\frac{\pi}{A}A_i \right)\\ 
&+\left(\frac{5\pi}{3}-2\alpha\right)
\cos\left(\frac{\pi}{A}A_i \right)-E(t)\sin^2(\alpha) 
\sin^2\left(\frac{\pi}{A}A_i\right)\\
&\, \geq 
\frac{5}{3}\sin\left(\frac{\pi}{A}A_i\right)
-\frac{5\pi}{3}\frac{A_i}{A}\cos\left(\frac{\pi}{A}A_i \right) 
 -E(t)\sin^2(\alpha) 
\sin^2\left(\frac{\pi}{A}A_i\right)\\
&\, \geq
\frac{5}{3} \left( 
\sin\left(\frac{\pi}{3}\right)
-\frac{\pi}{3}\cos\left(\frac{\pi}{3}\right) 
 \right) 
 -E(t)>0\,.\\
\end{align*}
\end{proof}
\noindent
Now that we have proven the monotonicity of the quantity $E(t)$ also when the two points 
$p$ and $q$ belong to a closed curve, taking into account that the computation 
in~\cite[Section $4$]{mantegazza_novaga_tortorelli2004}
is still valid for the other situation,
the two subsequent theorems follow as in~\cite{mantegazza_novaga_tortorelli2004} and 
in~\cite{magni_mantegazza_novaga2013}:
 
\begin{prop}\label{emb}
If $\Omega$ is bounded and strictly convex, there exists a constant $C>0$ depending only on 
$\Gamma_0$ such that $E(t)>C>0$ for every $t\in\left[0,T \right)$.
Hence, the networks $\Gamma_t$ remain embedded 
in all the maximal time interval of existence of the flow.
\end{prop}

\begin{lem}\label{emb2}
If $\Omega$ is strictly convex, the function
$E(\Gamma)$
defined on the class of $C^1$ networks without self-intersections (bounded or unbounded
and with or without end points or 3-points), is upper semi continuous with
respect to the $C^1_{loc}$ convergence.
Moreover, $E$ is dilation and translation invariant.
Consequently, every $C^1_{loc}$ limit $\Gamma_{\infty}$
of a sequence of rescaled networks $\widetilde{\Gamma}_t$ 
has no self-intersections, has multiplicity one 
and satisfies $E(\Gamma_{\infty})>C>0$, where the constant $C$ is given in Proposition~\ref{emb}
\end{lem}

\section{Monotonicity formula}
Now we introduce the Huisken's monotonicity formula, adapted to our situation.\\
For every $x_0\in\mathbb R^2$, we define the {\em backward heat kernel} relative to $(x_0,T)$ as
$$
\rho_{x_0}(x,t)=\frac{e^{-\frac{\vert
    x-x_0\vert^2}{4(T-t)}}}{\sqrt{4\pi(T-t)}}\,,
$$
and the {\em Gaussian density function} 
$\Theta:\Gamma\times\left[0,T \right)\rightarrow \mathbb{R}$
as
$$
\Theta\left(x_0,t \right)=\int_{\Gamma_t}\rho_{x_0}(x,t)ds\,. 
$$
We call $\widehat{\Theta}(x_0)$ the limit as $t \rightarrow T$ of $\Theta\left(x_0,t \right)$,
that exists and is finite (see \cite[Proposition 2.12]{magni_mantegazza_novaga2013}).\\
The following results can be obtained with minor modifications 
from~\cite[Proposition 6.4, Lemma 6.5]{mantegazza_novaga_tortorelli2004}.\\
\begin{prop}[Monotonicity Formula]
For every $x_0\in\mathbb{R}^2$ and $t\in[0,T)$ the following identity holds
\begin{align*}
\frac{d\,}{dt}\int_{{\Gamma_t}} \rho_{x_0}(x,t)\,ds=
&\,-\int_{{\Gamma_t}} \left\vert\,\underline{k}+\frac{(x-x_0)^{\perp}}{2(T-t)}\right\vert^2
\rho_{x_0}(x,t)\,ds\\
&\,+\left\langle\,\frac{P-x_0}{2(T-t)}\,\biggl\vert\,\tau(1,t)\right\rangle
\rho_{x_0}(P,t)\,,\nonumber
\end{align*}
where $\tau(1,t)$ is the unit tangent vector to $\gamma^2(x,t)$ in the point $P$.\\
Integrating between $t_1$ and $t_2$ with $0\leq t_1\leq t_2<T$ we get
\begin{align*}
\int_{t_1}^{t_2}\int_{{\Gamma_t}}
\left\vert\,\underline{k}+\frac{(x-x_0)^{\perp}}{2(T-t)}\right\vert^2
\rho_{x_0}(x,t)\,ds\,dt = 
&\,\int_{{\Gamma_{t_1}}} \rho_{x_0}(x,t_1)\,ds -
\int_{{\Gamma_{t_2}}} \rho_{x_0}(x,t_2)\,ds\\
&\,+\int_{t_1}^{t_2}
\left\langle\,\frac{P-x_0}{2(T-t)}\,\biggl\vert\,\tau(1,t)\right\rangle
\rho_{x_0}(P,t)\,dt\,.\nonumber
\end{align*}
\end{prop}

\begin{lem}\label{stimadib} Setting $\vert P- x_0\vert=d$, 
the following estimate holds 
\begin{equation*}
\left\vert\int_{t}^{T}\left\langle\,\frac{P-x_0}{2(T-\xi)}\,\biggl\vert\,
\tau(1,\xi)\right\rangle\rho_{x_0}(P,\xi)\,d\xi\,\right\vert\leq
\frac{1}{\sqrt{2\pi}}\int\limits_{d/\sqrt{2(T-t)}}^{+\infty}e^{-y^2/2}\,dy\leq
1/2\,.
\end{equation*}
Then, for every $x_0\in\mathbb R^2$,
\begin{equation*}
\lim_{t\to  T}\int_{t}^{T}
\left\langle\,\frac{P-x_0}{2(T-\xi)}\,\biggl\vert\,\tau(1,\xi)\right\rangle
\rho_{x_0}(P,\xi)\,d\xi=0\,.
\end{equation*}
\end{lem}

\noindent
We rescale now the flow in its maximal time interval $\left[0,T \right)$.\\
Fixed $x_0\in\mathbb R^2$, let $\widetilde{F}_{x_0}:\Gamma\times
[-1/2\log{T},+\infty)\to\mathbb R^2$ be the map
$$
\widetilde{F}_{x_0}(p,\mathfrak t)=\frac{F(p,t)-x_0}{\sqrt{2(T-t)}}\qquad
\mathfrak t(t)=-\frac{1}{2}\log{(T-t)}
$$
then, the rescaled networks are given by 
$$
\widetilde{\Gamma}_{x_0,\mathfrak t}=\frac{\Gamma_t-x_0}{\sqrt{2(T-t)}}
$$
and they evolve according to the equation 
$$
\frac{\partial\,}{\partial
  \mathfrak t}\widetilde{F}_{x_0}(p,\mathfrak t)=\widetilde{\underline{v}}(p,\mathfrak t)+\widetilde{F}_{x_0}(p,\mathfrak t)
$$
where 
$$
\widetilde{\underline{v}}(p,\mathfrak t)=\sqrt{2(T-t(\mathfrak t))}\cdot\underline{v}(p,t(\mathfrak t))=
\widetilde{\underline{k}}+\widetilde{\underline{\lambda}}
=\widetilde{k}\widetilde{\nu}+\widetilde{\lambda}\widetilde{\tau}
\; \text{ and }\;
t(\mathfrak t)=T-e^{-2\mathfrak t}\,.
$$
Hence we obtained a smooth flow of spoon-shaped networks $\widetilde{\Gamma}_{\mathfrak t}$
defined for $\mathfrak t\in [-\frac12 \log T, +\infty)$.\\
By a straightforward computation
(\cite{huisken1990},~\cite[Lemma~6.7]{mantegazza_novaga_tortorelli2004}) 
we have the following
rescaled version of the monotonicity formula.

\begin{prop}[Rescaled Monotonicity Formula]
Let $x_0\in\mathbb R^2$ and set  
$$
\widetilde{\rho}(x)=e^{-\frac{\vert x\vert^2}{2}}\,.
$$
For every $\mathfrak t\in[-1/2\log{T},+\infty)$ the following identity holds
\begin{equation*}
\frac{d\,}{d\mathfrak t}\int_{\widetilde{\Gamma}_{x_0,\mathfrak t}}
\widetilde{\rho}(x)\,d\sigma=
-\int_{\widetilde{\Gamma}_{x_0,\mathfrak t}}\vert
\,\widetilde{\underline{k}}+x^\perp\vert^2\widetilde{\rho}(x)\,d\sigma
+\left\langle\,{\widetilde{P}_{x_0,\mathfrak t}}
\,\Bigl\vert\,{\tau}(1,t(\mathfrak t))\right\rangle
\widetilde{\rho}(\widetilde{P}_{x_0,\mathfrak t})
\end{equation*}
where $\widetilde{P}_{x_0,\mathfrak t}=\frac{P-x_0}{\sqrt{2(T-t(\mathfrak t))}}$.\\  
Integrating between $\mathfrak t_1$ and $\mathfrak t_2$ with 
$-1/2\log{T}\leq \mathfrak t_1\leq \mathfrak t_2<+\infty$ we get
\begin{align} 
\int_{\mathfrak t_1}^{\mathfrak t_2}\int_{\widetilde{\Gamma}_{x_0,\mathfrak t}}\vert
\,\widetilde{\underline{k}}+x^\perp\vert^2\widetilde{\rho}(x)\,d\sigma\,d\mathfrak t=
&\, \int_{\widetilde{\Gamma}_{x_0,\mathfrak t_1}}\widetilde{\rho}(x)\,d\sigma - 
\int_{\widetilde{\Gamma}_{x_0,\mathfrak t_2}}\widetilde{\rho}(x)\,d\sigma\label{reseqmonfor-int}\\
&\,\int_{\mathfrak t_1}^{\mathfrak t_2}\left\langle\,{\widetilde{P}_{x_0,\mathfrak t}}\,
\Bigl\vert\,{\tau}(1,t(\mathfrak t))\right
\rangle\widetilde{\rho}(\widetilde{P}_{x_0,\mathfrak t})\,d\mathfrak t\,.\nonumber
\end{align}
\end{prop}
\noindent
Consequently, we have the analogue of Lemma~\ref{stimadib}, whose proof follows
in the same way, substituting the rescaled quantities.

\begin{lem}\label{rescstimadib}
The following estimate holds 
\begin{equation*}
\left\vert\int_{\mathfrak t}^{+\infty}\left\langle\,{\widetilde{P}_{x_0,\xi}}\,
\Bigl\vert\,{\tau}(1,t(\xi))\right
\rangle\widetilde{\rho}(\widetilde{P}_{x_0,\xi})\,d\xi\right\vert\leq \sqrt{\pi/2}\,.
\end{equation*}
Then, for every $x_0\in\mathbb R^2$,
\begin{equation*}
\lim_{\mathfrak t\to  +\infty}\int_{\mathfrak t}^{+\infty}\left\langle\,{\widetilde{P}_{x_0,\xi}}\,
\Bigl\vert\,{\tau}(1,t(\xi))\right
\rangle\widetilde{\rho}(\widetilde{P}_{x_0,\xi})\,d\xi=0\,.
\end{equation*}
\end{lem}

\section{Proof of Theorem \ref{thmain}}
\begin{defn}
We call \emph{infinite flat triod} a connected planar set composed by 
three half-lines intersecting in the origin of $\mathbb{R}^2$, forming angles of $120$ degrees.
\end{defn}

\begin{defn}\label{BS}
We call \emph{Brakke spoon} a spoon-shaped network
which shrinks in a self-similar way during the evolution.
\end{defn}

\begin{center}
\begin{tikzpicture}[scale=.75]
\draw[color=black!30!white]
(-6.25,0) to[out= 0,in=-180, looseness=1] (-4.5,0);
\node[inner sep=0pt] at (-4,0)
    {\includegraphics[width=.375\textwidth]{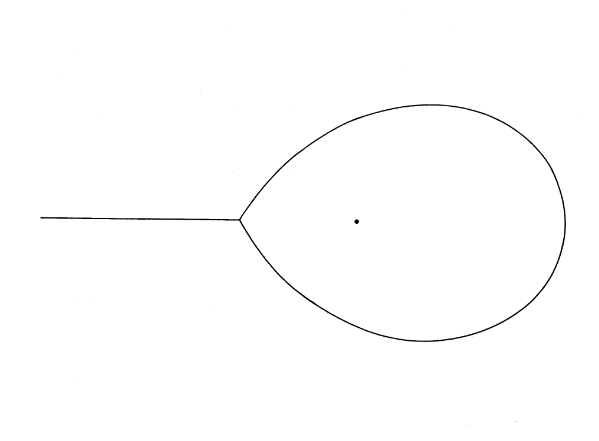}};
\draw[color=black!30!white]
(1.25,0)to[out= 135,in=-45, looseness=1] (1.1,0.15)
(1.25,0)to[out= -135,in=45, looseness=1] (1.1,-0.15)
(-3.25,3)to[out= -45,in=135, looseness=1] (-3.1,2.85)
(-3.25,3)to[out= 45,in=-135, looseness=1] (-3.4,2.85)
(-3.25,-3) to[out= 90,in=-90, looseness=1] (-3.25,3)
(-4.75,0) to[out= 0,in=-180, looseness=1] (1.25,0);
\path[font=\large]
 (-3,-0.1) node[above] {$O$};
\end{tikzpicture}
\end{center}
\begin{fig}
A Brakke spoon.
\end{fig}

\noindent 
A Brakke spoon is composed by a half-line which intersects a closed curve, 
forming angles of $120$ degrees and
satisfying the equation $k+\left\langle x\vert \nu \right\rangle =0$ 
for all point $x$ of the curve, where $k$ is the curvature and $\nu$
the unit normal vector in the point $x$.
This configuration was mentioned in~\cite{brakke}
as an example of evolving network with a loop shrinking 
down to a point, leaving a half-line that then vanishes instantaneously.\\ 
If we require that this particular spoon-shaped network 
is embedded, with multiplicity $1$, without self-intersections
and with $120$ degrees angles at the triple junction, then
it is unique, up to translation, rotation and dilation.

\begin{prop}\label{resclimit}
Assume that the length of the non-closed curve $\gamma^2(x,t)$ is
uniformly bounded away from zero for $t\in[0,T)$.
Then, for every $x_0\in\mathbb R^2$ and for every subset $\mathcal I$ of  $[-1/2\log
T,+\infty)$ with infinite Lebesgue measure, 
there exists a sequence of rescaled times $\mathfrak t_j\to+\infty$,  
with $\mathfrak t_j\in{\mathcal I}$, such that the sequence of rescaled networks
$\widetilde{\Gamma}_{x_0,\mathfrak t_{j}}$ converges in the $C^1_{\rm loc}$ topology
to a limit set 
$\Gamma_\infty$
which, {\em if not empty}, is one of the following:
\begin{itemize}
\item a half-line from the origin with multiplicity one
 (in this case $\widehat\Theta(x_0)=1/2$).
\item a straight line through the origin with multiplicity one (in this case $\widehat\Theta(x_0)=1$);
\item an infinite flat triod centered at the origin with multiplicity
  one (in this case $\widehat\Theta(x_0)=3/2$);
\item a Brakke spoon  with multiplicity one (in this case $\widehat\Theta(x_0)>3/2$). 
\end{itemize}
Moreover, for every sequence of rescaled networks 
$\widetilde{\Gamma}_{x_0,\mathfrak t_j}$ converging at
least in the $C^1_{\rm loc}$ topology to a limit ${\Gamma}_\infty$,
as $\mathfrak t_j\to+\infty$, we have
\begin{equation}\label{gggg}
\lim_{j\to\infty}\frac{1}{\sqrt{2\pi}}\int_{\widetilde{\Gamma}_{x_0,\mathfrak t_j}}\widetilde{\rho}\,
d\sigma=\frac{1}{\sqrt{2\pi}}
\int_{\Gamma_\infty}\widetilde{\rho}\,d\sigma=\widehat{\Theta}(x_0)\,.
\end{equation}
\end{prop}

\begin{proof}
We assume that we have a sequence of rescaled networks $\widetilde{\Gamma}_{x_0,\mathfrak t_j}$ 
that converges to a limit $\Gamma_\infty$ as $j \rightarrow \infty$ in the $C^1_{loc}$ topology.
As in the proof of Proposition~2.19
in~\cite{magni_mantegazza_novaga2013}, 
we can pass to the limit in the following Gaussian integral:
$
\lim_{j\to\infty}\frac{1}{\sqrt{2\pi}}\int_{\widetilde{\Gamma}_{x_0,\mathfrak t_j}}\widetilde{\rho}\,d\sigma=\frac{1}{\sqrt{2\pi}}
\int_{\Gamma_\infty}\widetilde{\rho}\,d\sigma\,.
$
Consequently, 
$$
\frac{1}{\sqrt{2\pi}}\int_{\widetilde{\Gamma}_{x_0,\mathfrak t_j}}\widetilde{\rho}\,d\sigma
=\int_{\Gamma_{t(\mathfrak t_j)}}\rho_{x_0}(x,t(\mathfrak t_j))\,ds
=\Theta(x_0,t(\mathfrak t_j))\to\widehat{\Theta}(x_0)\,,
$$
as $j\to\infty$ and equality~\eqref{gggg} follows.\\
If in the rescaled monotonicity formula~\eqref{reseqmonfor-int} 
we set $\mathfrak t_1=-\frac12 \log T$ and we let $\mathfrak t_2$ go to $+\infty$,
considering the result in Lemma~\ref{rescstimadib},
we get
$
\int^{+\infty} _{-1/2 \log T} \int_{\widetilde{\Gamma}_{x_0,\mathfrak{t}}}
\vert\widetilde{k}+\left\langle x \vert \nu \right\rangle
\vert^2  \widetilde{\rho} \, d\sigma \, d\mathfrak{t}<+\infty\,, 
$
then
$
\int_{\mathcal I} \int_{\widetilde{\Gamma}_{x_0,\mathfrak{t}}}
\vert\widetilde{k}+\left\langle x \vert \nu \right\rangle
\vert^2  \widetilde{\rho} \, d\sigma \, d\mathfrak{t}<+\infty\,.
$
Being the last integral finite and being the integrand a nonnegative 
function on a set of infinite Lebesgue measure,
we can extract within $\mathcal{I}$ a sequence of time $\mathfrak t_j\rightarrow+\infty$
such that
$$
\lim_{j\rightarrow +\infty} \int_{\widetilde{\Gamma}_{x_0,\mathfrak{t}_j} }
\vert\widetilde{k}+\left\langle x \vert \nu \right\rangle
\vert^2  \widetilde{\rho} \, d\sigma =0 
$$
It follows that , for every ball of radius $R$ in $\mathbb{R}^2$,
the networks $\widetilde{\Gamma}_{x_0,\mathfrak{t_j}}$ 
have curvature uniformly bounded in $L^2(B_R)$.\\
Hence, we can extract a subsequence
$\widetilde{\Gamma}_{x_0,\mathfrak{t} _{ j}}$ (not relabelled) which, after a 
possible reparametrization, converges to a limit (possibly empty) network
$\Gamma_{\infty}$ in the $C^1_{loc}$ topology.\\
We will call $\Gamma_{\infty}$ a blow-up limit of $\Gamma_t$
around the point $x_0$
(notice that the blow-up limit could depend on the choice of the subsequence).\\
If the limit set is not empty, it has multiplicity one thanks to Lemma~\ref{emb2}. 
Also the embeddedness is granted by Proposition~\ref{emb}.\\
This limit set satisfies the equation 
\begin{equation}\label{omotetiche}
k_\infty+\left\langle x\vert \nu \right\rangle =0\; 
\text{for all}\; x\in \Gamma_\infty\,,
\end{equation}
where $k_\infty$ is the curvature of $\Gamma_\infty$ at $x$,
because the integral 
$\int_{\widetilde{\Gamma}} \vert \widetilde{k}+
\left\langle x \vert \nu \right\rangle
\vert^2 \widetilde{\rho} \, d \sigma$ is lower semicontinuous
under the $C^1_{loc}$ convergence.\\
In principle, the limit set is composed of curves in $W^{2,2}_{loc}$,
but from the relation $k_\infty+\left\langle x \vert \nu \right\rangle =0$
it follows that $k_\infty$ is continuous, since the curves are in $C^1_{loc}$.
By a bootstrap argument we gain the smoothness of  $\Gamma_{\infty}$.\\
Now we classify all the possible limit sets that satisfy~\ref{omotetiche}.\\ 
By the work of Abresch and Langer \cite{abresch_langer}
the curves componing $\Gamma_{\infty}$ can be only lines, half-lines, 
segments or pieces of curves of Abresch and Langer.
If $k$ is identically zero we can only
have a straight line, a half-line or an infinite flat triod.
If the point $x_0\in \mathbb{R}^2$ is not the end point $P$, then $\Gamma_\infty$
has no end points, otherwise
if the point $x_0\in \mathbb{R}^2$ is the end-point $P$, the set $\Gamma_\infty$ 
is a half-line by the bound from below on the curve $\gamma^2$.\\
Otherwise, if the curvature is not constantly zero,
it is proved in~\cite[Theorem~3]{chen_guo} that
there exists a unique possible limit set without self-intersection
and with multiplicity one and it is  the Brakke spoon in Definition~\ref{BS}.\\
The value of gaussian density $\widehat{\Theta}(x_0)$ in the first three cases
can be obtained by a direct computation (see~\cite{magni_mantegazza_novaga2013}), 
for the last case see Lemma~8.4 in~\cite{ilmanen_neves_schulze2014}.\\
\end{proof}

\noindent We define the set of reachable points of the flow by
$$
R=\left\lbrace x\in \mathbb{R}^2 \vert\; \text{there exist}\, p_i\in \Gamma \;
\text{and} \,t_i\nearrow T \; \text{such that}
\lim_{i\longrightarrow\infty}F(p_i,t_i)=x \right\rbrace\,. 
$$

\noindent This set $R$ is not empty and compact,
if a point $x_0\notin R$, it means that the flow is definitively far from $x_0$, otherwise 
if $x_0\in R$, for every $t\in [0,T)$ the closed ball of radius $\sqrt{2(T-t)}$ and center $x_0$ 
intersects $\Gamma_t$
(for the proof see~\cite{magni_mantegazza2014}).
Hence we have two possibilities when we consider the blow-up around points of $\overline{\Omega}$:
\begin{itemize}
\item  we are rescaling around a point in $R$ and
the limit of any sequence of rescaled network is not empty;
\item the blow-up limit is empty.
\end{itemize}

\noindent Fixed $x_0\in R$, repeating the previous arguments, there exists 
a not relabeled subsequence $\mathfrak t_j\to+\infty$ of rescaled times such that the rescaled networks 
$\widetilde{\Gamma}_{x_0,\mathfrak t_{j}}$ converge in the $C^1_{loc}$ topology to a nonempty limit:
a straight line, a half-line, an infinite flat triod or a Brakke spoon.

\begin{prop}
If the sequence of rescaled networks converges to a straight line or to a half-line,
then the curvature of the evolving network is uniformly bounded for $t\in [0,T)$
in a ball around the point $x_0$.
If the sequence of rescaled networks converges to a infinite flat triod,
then the full $L^2$ norm of the curvature of the evolving network is bounded.
\end{prop} 

\begin{proof}
See~\cite{magni_mantegazza_novaga2013}, Propositions~3.2 and 3.3 and the proof of Theorem~2.4.
\end{proof} 

{\it Proof of Theorem \ref{thmain}.\;}
As we have imposed that the length of the non-closed curve $\gamma^2$ 
stays bounded away from zero during the evolution,
at the time $T$ of maximal existence of a smooth flow, only the second
case of Theorem~\ref{infinitecurvature} can happen, that is 
$\limsup_{t\rightarrow T} \int_{\Gamma_t}k^2 ds = +\infty\,.$
By Proposition~\ref{resclimit} we know that, for every $x_0\in\mathbb{R}^2$,
there exists a sequence of rescaled times such that the sequence of
corresponding rescaled networks converges
to a limit set $\Gamma_\infty$, which, if not empty, 
is a half-line from the origin, a straight line through the origin, an infinite flat triod centered 
at the origin or a Brakke spoon.
If for every $x_0\in R$ the limit set $\Gamma_\infty$ is a half-line, a 
straight line or an infinite flat triod, then the curvature is uniformly 
bounded as $t\to T$ thanks to the previous proposition.
This is in contradiction with Theorem~\ref{infinitecurvature}.
Hence, it exists a point $x_0\in R$ such that the limit set $\Gamma_\infty$ is a Brakke spoon and at $x_0$ the gaussian density is
$\widehat{\Theta}(x_0)>3/2$.
This proves the main Theorem.
\qed

\begin{cor}
If the length of the non-closed curve $\gamma^2$ is positively bounded 
from below,
at the maximal time of existence $T<+\infty$ of a smooth solution
$\lim_{t \rightarrow T} A(t)=0 $, $\liminf_{t \rightarrow T} L_1(t)=0$
and $\limsup_{t \rightarrow T} \int_{\gamma^1} k^2\,ds=+\infty\,.$
Moreover $T=\frac{3A_0}{5\pi}$ where $A_0$ is the initial area of the region contained in the closed curve $\gamma^1$.
\end{cor}

\begin{proof}
From the equation~\eqref{area} of the evolution of the area
enclosed in the loop, we know that $A(t)$ is decreasing.
If the length of the non-closed curve $\gamma^2$ stays positively bounded 
from below during all the smooth evolution of the flow and we know that at time $T<+\infty$
it appears a singularity, then from Theorem \ref{thmain}, the rescaled flow 
converges, up to subsequence, to a Brakke spoon.
This implies that the area in the loop goes to zero and 
$\liminf_{t \rightarrow T} L_1(t)=0$.
Recalling the inequality \eqref{curvlen}
we also get that $\limsup_{t \rightarrow T} \int_{\gamma^1} k^2\,ds=+\infty\,.$
Moreover, thanks to the equation
$
\frac{\partial A}{\partial t}=\frac{-5\pi}{3}\,,
$ 
the time in which such singularity appears is given by 
\begin{equation*}\label{singtime}
T=\frac{3A_0}{5\pi}\,.
\end{equation*}
\end{proof}


\begin{thebibliography}{0} 
\bibitem{abresch_langer} {\sc U.~Abresch and J.~Langer}, {\em The normalized curve shortening flow
and homothetic solutions}, J. Diff. Geom. \textbf{23} (1986), 175-196.
\bibitem{brakke}{\sc K.~A.~Brakke}, {\em The Motion of a Surface by its Mean Curvature}, 
Princeton University Press, Princeton, N,J , 1978.
\bibitem{bronsard_reitich1993} {\sc L.~Bronsard and F.~Reitich}, {\em On three-phase boundary 
motion and the singular limit of a vector-valued Ginzburg-Landau equation}, 
Arch. Rational Mech. Anal. \textbf{124} (1993), 355-379.
\bibitem{chen_guo} {\sc X.~Chen and J.~Guo}, {\em Motion by curvature of planar curves with end 
points moving freely on a line}, Math. Ann., \textbf{350} (2011), 277-311. 
\bibitem{chzh} {\sc K.~S.~Chou and X.~P.~Zhu}, {\em Shortening complete planar curves},
 J. Diff. Geom. \textbf{50} (1998), 471-504.
\bibitem{gage_hamilton1986} {\sc M.~E.~Gage and R.~S.~Hamilton}, {\em The heat equation shrinking
convex plane curves}, J. Diff. Geom. \textbf{23} (1986), 69-96.
\bibitem{grayson1987} {\sc M.~A.~Grayson}, {\em The heat equation shrinks embedded plane curves to 
points}, J. Diff. Geom. \textbf{26} (1987), 285-314.
\bibitem{hamilton} {\sc R.~S.~Hamilton}, {\em Four-manifolds with positive curvature operator},
J. Diff.Geom. \textbf{24} (1986), 153-179.
\bibitem{hamilton3} {\sc R.~S.~Hamilton}, {\em Isoperimetric estimates for the curve shrinking flow in 
the plane}, Modern Methods in Complex Analysis (Princeton, NJ, 1992), Princeton University Press, 
Princeton, NJ, 1995, 201-222. 
\bibitem{huisken1990} {\sc G.~Huisken}, {\em Asymptotic behavior for singularities of the mean 
curvature flow}, J. Diff. Geom. \textbf{31} (1990), 285-299.
\bibitem{huisk2} {\sc G.~Huisken}, {\em A distance comparison principle for evolving curves}, 
Asian J. Math. \textbf{2} (1998), 127-133.
\bibitem{ilmanen_neves_schulze2014}{\sc T.~Ilmanen, A.~Neves and F.~Schulze}, 
{\em On short time existence for the planar network flow}, 2014, arXiv:1407.4756
\bibitem{kinderlehrer_liu}{\sc D.~Kinderlehrer and C.~Liu}, {\em Evolution of grain boundaries}, Math. 
Models Methods Appl. Sci. \textbf{11} (2001), no. 4,713-729.
\bibitem{magni_mantegazza2014} {\sc  A.~Magni and  C.~Mantegazza}, 
{\em  A note on Grayson's theorem}, Rend. Semin. Mat. Univ. Padova \textbf{131} (2014), 263-279.
\bibitem{magni_mantegazza_novaga2013} {\sc A.~Magni, C.~Mantegazza and M.~Novaga}, 
{\em Motion by curvature of planar networks II}, 2013, arXiv:1301.3352v1.
\bibitem{mantegazza_novaga_pluda_schulze2015} {\sc  C.~Mantegazza, M.~Novaga, A.~Pluda and 
F.~Schulze}, {\em Evolution of network with multiple junctions}, Preprint 2015.
\bibitem{mantegazza_novaga_tortorelli2004} {\sc C.~Mantegazza, M.~Novaga and V.~M.~Tortorelli}, 
{\em Motion by curvature of planar networks}, Ann. Sc. Norm. Super. Pisa Cl. Sci. \textbf{(5)3}
(2004), 235-324.
\bibitem{nirenberg59} {\sc L.~Nirenberg}, 
{\em On elliptic partial differential equations}, Ann. Sc. Norm. Super. Pisa Cl. Sci. \textbf{13}
(1959), 116-162.

\end{thebibliography}
\end{document}